\documentclass[12pt]{article}
\usepackage{amsthm,amssymb,amsmath,fullpage,xspace,multirow,verbatim}

\newcounter{thmctr}
\newtheorem{thm}[thmctr]{Theorem}
\newtheorem{lemma}[thmctr]{Lemma}
\newtheorem{prop}[thmctr]{Proposition}
\newtheorem{cor}[thmctr]{Corollary}
\newtheorem{conj}[thmctr]{Conjecture}

\newcommand{\dhruvuni}{University of Illinois at Chicago \\ mubayi@math.uic.edu}
\newcommand{\johnuni}{University of Illinois at Chicago \\ lenz@math.uic.edu}
\newcommand{\dhruvfoot}{\footnote{Research supported in part by  NSF Grants 0969092 and 1300138.}}
\newcommand{\johnfoot}{\footnote{Research partly supported by NSA Grant H98230-13-1-0224.}}

\ifx\buildsubmitted\undefined
    \newenvironment{notarxiv}{\comment}{\endcomment}
    \newenvironment{onlyarxiv}{}{}
\else
    \newenvironment{notarxiv}{}{}
    \newenvironment{onlyarxiv}{\comment}{\endcomment}
\fi

\title{Multicolor Ramsey Numbers for Complete Bipartite Versus Complete Graphs}
\author{John Lenz \johnfoot \\ \johnuni \and Dhruv Mubayi \dhruvfoot \\ \dhruvuni}

\begin{document}
\maketitle
\begin{abstract}
  Let $H_1, \ldots, H_{k}$ be graphs.  The multicolor Ramsey number
  $r(H_1,\ldots,H_{k})$ is the minimum integer $r$ such that in every edge-coloring of $K_r$ by
  $k$ colors, there is a monochromatic copy of $H_i$ in color $i$ for some $1 \leq i \leq k$.
  In this paper, we investigate the multicolor Ramsey number $r(K_{2,t},\ldots,K_{2,t},K_m)$,
  determining the asymptotic behavior up to a polylogarithmic factor for almost all ranges of $t$ and $m$.
  Several different constructions are used for the lower bounds, including the random graph and
  explicit graphs built from finite fields.  A technique of Alon and R\"odl using the probabilistic
  method and spectral arguments is employed to supply tight lower bounds.  A sample result is
  $$c_1 \frac{m^2t}{\log^4(mt)} \leq r(K_{2,t},K_{2,t},K_m) \leq c_2 \frac{m^2t}{\log^2 m}$$ for any
  $t$ and $m$, where $c_1$ and $c_2$ are absolute constants.

  \medskip

\end{abstract}

\section{Introduction} 

The multicolor Ramsey number $r(H_1,\dots,H_{k})$ is the minimum integer $r$ such that in every
edge-coloring of $K_r$ by $k$ colors, there is a monochromatic copy of $H_i$ in color $i$ for some
$1 \leq i \leq k$.  Ramsey's famous theorem~\cite{ram-ramsey30} states that $r(K_s,K_t) < \infty$
for all $s$ and $t$.  Determining these numbers is usually a very difficult problem.  Even
determining the asymptotic behavior is difficult; there are only a few infinite families of graphs
where the order of magnitude is known.  A famous example is $r(K_3,K_m) = \Theta(m^2/\log m)$, where
the upper bound was proved by Ajtai, Koml\'os, and Szemer\'edi~\cite{ram-ajtai80} and the lower
bound by Kim~\cite{ram-kim95}.

For more colors, in 1980 Erd\H{o}s and S\'os~\cite{ram-erdos80} conjectured
that $r(K_3,K_3,K_m)/r(K_3,K_m) \rightarrow \infty$ as 
$m \rightarrow \infty$.  This conjecture was open for 25 years until it was proved true
by Alon and R\"odl~\cite{mr-alon05}.  In their paper, they provided a general technique using
graph eigenvalues and the probabilistic method which provides good
estimates on multicolor Ramsey numbers.  This breakthrough provided the first
sharp asymptotic (up to a poly-log factor) bounds on infinite families of multicolor Ramsey numbers with at
least three colors.

The exact results proved by Alon and R\"odl~\cite{mr-alon05} are shown in Table~\ref{alontable}.
For $k \geq 1$, define $r_k(H;G)$ to be $r(H,\dots,H,G)$, where $H$ is repeated $k$ times.  In other
words, $r_k(H;G)$ is the minimum integer $r$ such that in every edge-coloring of $K_r$ by $k+1$
colors, there is a monochromatic copy of $H$ in one of the first $k$ colors or a copy of $G$ in the
$k+1$st color.  In Table~\ref{alontable}, $s$ and $t$ are fixed with $t \geq (s-1)! + 1$, $\delta >
0$ is any positive constant, and $m$ is going to infinity.  Also, in the tables below, $a \ll b$
means there exists some positive constant $c$ such that $a \leq c b$.
All logarithms in this paper are base $e$.

\begin{table}
  \newcommand\T{\rule{0pt}{3ex}}
  \newcommand\B{\rule[-1.8ex]{0pt}{0pt}}
  \begin{center}
    \begin{tabular}{| c | c | c |} 
      \hline
      $H$ & $k = 2$ & $k \geq 3$ \\ \hline
      $K_3$    & $\frac{m^{3}}{\log^{4+\delta} m} \ll r_2(K_3;K_m) 
                 \ll \frac{m^3 \log\log m}{\log^2 m}$ 
               & $\frac{m^{k+1}}{\log^{2k+\delta} m} \ll r_k(K_3;K_m)
                 \ll \frac{m^{k+1}\left( \log \log m \right)^{k-1}}{\log^k m}$ 
  \B \T \\ \hline
      $C_4$    & $\frac{m^2}{\log^4 m} \ll r_2(C_4;K_m) \ll \frac{m^2}{\log^2 m}$  
               & $r_k(C_4;K_m) = \Theta\left( \frac{m^{2}}{\log^{2} m} \right)$ 
  \B \T \\ \hline
      $C_6$    & $\frac{m^{3/2}}{\log^3 m} \ll r_2(C_6;K_m) \ll \frac{m^{3/2}}{\log^{3/2} m}$  
               & $r_k(C_6;K_m) = \Theta\left( \frac{m^{3/2}}{\log^{3/2} m} \right)$ 
  \B \T \\ \hline
      $C_{10}$ & $\frac{m^{5/4}}{\log^{5/2} m} \ll r_2(C_{10};K_m) \ll \frac{m^{5/4}}{\log^{5/4} m}$  
               & $r_k(C_{10};K_m) = \Theta\left( \frac{m^{5/4}}{\log^{5/4} m} \right)$
  \B \T \\ \hline
      $K_{s,t}$ & $\frac{m^s}{\log^{2s} m} \ll r_2(K_{s,t};K_m) \ll \frac{m^s}{\log^s m}$  
                & $r_k(K_{s,t};K_m) = \Theta\left( \frac{m^{s}}{\log^{s} m} \right)$ 
  \B \T \\ \hline
    \end{tabular}
    \caption{ Results on $r_k(H;K_m)$ proved by Alon and R\"odl~\cite{mr-alon05}.} \label{alontable}
  \end{center}
\end{table}

One surprising aspect of Alon and R\"odl's~\cite{mr-alon05} techniques is that they prove very good
upper and lower bounds for multicolor Ramsey numbers in cases where the two-color Ramsey number is
not as well understood.  For example, Erd\H{o}s~\cite{ram-erdos83} conjectured that $r(C_4,K_m) =
O(m^{2-\epsilon})$ for some absolute constant $\epsilon > 0$, and this conjecture is still open.
The current best upper bound is an unpublished result of Szemer\'edi which was reproved by Caro,
Rousseau, and Zhang~\cite{ram-caro00} where they showed that $r(C_4,K_m) = O(m^2/\log^2 m)$ and the
current best lower bound is $\Omega(m^{3/2}/\log m)$ by Bohman and Keevash \cite{ram-bohman11}.  In
sharp contrast, for three colors Alon and R\"odl~\cite{mr-alon05} determined $r(C_4,C_4,K_m)$ up to
a poly-log factor and found the order of magnitude of $r_k(C_4;K_m)$ for $k \geq 3$.  A similar
situation occurs for the other graphs in Table~\ref{alontable} besides $K_3$.

\section{Results} 

We focus on the problem of determining $r_k(K_{2,t};K_m)$ when $k$ is fixed and $t$ is no longer a
constant.  Our results can be summarized by the following table; more precise statements are given
later.

\begin{table}[ht]
  \newcommand\T{\rule{0pt}{3ex}}
  \newcommand\B{\rule[-1.8ex]{0pt}{0pt}}
  \begin{center}
    \begin{tabular}{| c | c | c | c |}
      \hline
       & $m \ll \log^2 t$ & $\log^2 t \ll m \ll 2^t$ & $2^t \ll m$
  \T \B \\ \hline
       $k = 1$
               & $mt \ll r \ll \frac{m^2 t}{\log^2 m}$
               & $\frac{m^2 t}{\log^2 (mt)} \ll r \ll \frac{m^2 t}{\log^2 m}$
               & $r \ll \frac{m^2 t}{\log^2 m}$
  \T \B \\ \hline
       $k = 2$ 
               & $mt \ll r \ll \frac{m^2 t}{\log^2 m}$
               & $\frac{m^2 t}{\log^2 (mt)} \ll r \ll \frac{m^2 t}{\log^2 m}$
               & $\frac{m^2 t}{\log^4 (mt)} \ll r \ll \frac{m^2 t}{\log^2 m}$
  \T \B \\ \hline
       $k \geq 3$
               & $mt \ll r \ll \frac{m^2 t}{\log^2 m}$
               & $\frac{m^2 t}{\log^2 (mt)} \ll r \ll \frac{m^2 t}{\log^2 m}$
               & $\frac{m^2 t}{\log^2 (mt)} \ll r \ll \frac{m^2 t}{\log^2 m}$
  \T \B \\ \hline
    \end{tabular}
  \end{center}
  \caption{Results on $r = r_k(K_{2,t};K_m)$ in this paper.} \label{ourtable}
\end{table}

We are able to find the order of magnitude of $r_k(K_{2,t};K_m)$ up to a ploy-log factor for all
ranges of $m$ and $t$ except the upper right table cell where $m$ is much larger than $t$ and $k =
1$.  This is similar to the fact that the order of magnitude of $r(C_4,K_m)$ is unknown but Alon and
R\"odl~\cite{mr-alon05} found the order of magnitude up to a poly-log factor when $k \geq 2$.  So
the only remaining case is $r(K_{2,t},K_m)$ when $m$ is much larger than $t$.  The best known lower
bound is $r(K_{2,t},K_m) \geq c_t (m/\log m)^{\rho(K_{2,t})}$, where $\rho(K_{2,t}) = 2 -
\frac{2}{t}$ (see \cite{prob-alonspencer,ram-krivelevich95}.) Unfortunately, this lower bound has a
constant $c_t$ depending on $t$ when we would like to know the exact order of magnitude.

The upper bound in Table~\ref{ourtable} is a straightforward counting argument using the extremal
number of $K_{2,t}$.   Szemer\'edi (unpublished) and Caro, Rousseau, and Zhang~\cite{ram-caro00}
proved the following proposition for two colors; we extend it for all $k$ using a related but
slightly different technique.

\begin{prop}\label{ramupper}
  For $k \geq 1$, $t \geq 2$, and $m \geq 3$ integers, there exists a constant $c$ depending only on $k$ such that
  \begin{align*}
    r_k(K_{2,t};K_m) \leq c \frac{m^2t}{\log^2 m}.
  \end{align*}
\end{prop}

\newcommand{\reflinearlower}{Proposition~\ref{thmlinearlower}\xspace}
\newcommand{\refrandlower}{Proposition~\ref{thmrandlower}\xspace}
\newcommand{\refktwolower}{Theorem~\ref{alllowerthm}~\textit{(i)}\xspace}
\newcommand{\refkthreelower}{Theorem~\ref{alllowerthm}~\textit{(ii)}\xspace}

The main contribution in this paper is the various lower bounds given in the table.  One simple
lower bound is to take $m-1$ vertex sets $X_1, \ldots, X_{m-1}$, each of size $t+1$.  Color
edges inside each $X_i$ with one color and color all edges between $X_i$s in the other color.
This proves $r(K_{2,t},K_m) > (m-1)(t+1)$.  In fact, this proves the following proposition.

\begin{prop} \label{thmlinearlower}
  Let $k \geq 1$, $t \geq 2$, and $m \geq 3$ be integers. Then $r_k(K_{2,t};K_m) > (m-1)(t+1)$.
\end{prop}

Note that being slightly more clever for $k \geq 2$ and making each $X_i$ of size $r_k(K_{2,t}) - 1$
does not give a large improvement.  A theorem of Lazebnik and Mubayi~\cite{ffc-lazebnik02} proves
that $r_k(K_{2,t}) > k^2(t-1)$ when $k$ and $t$ are prime powers and $r_k(K_{2,t}) \leq k^2 (t-1)
+ k + 2$ for all $k$ and $t$.  Therefore the size of each $X_i$ could be increased to roughly $k^2
t$ but that implies only a constant improvement in \reflinearlower.

Another lower bound comes from the random graph $G(n,p)$.  Consider a coloring of $E(K_n)$ obtained
by taking $k$ random graphs $G(n,p)$ as the first $k$ colors and letting the last color be the
remaining edges.  Depending on the choice of $n$ and $p$, this construction avoids $K_{2,t}$ in the
first $k$ colors and $K_m$ in the last color. In \refrandlower, we show that when $\log^2 t \ll m
\ll 2^t$ it is possible to choose $p$ so that $G(m^2t/\log^2 (mt),p)$ avoids $K_{2,t}$ and has
independence number at most $m$.  When $m \gg 2^t$, the number of vertices must be reduced to
roughly $m^{2-2/t}$ which does not provide a good lower bound on the Ramsey number.  Most likely,
when $m \ll \log^2 t$ a more detailed analysis shows that one can choose $p$ so that $G(m^2t/\log^2
(mt),p)$ avoids $K_{2,t}$ and has independence number at most $m$.  We skip this analysis and only
investigate Proposition~\ref{thmrandlower} for $m \gg \log^2 t$ because when $m \ll \log^2 t$, the
lower bound of $mt$ from \reflinearlower is better than $m^2t/\log^2(mt)$.  The precise statement of
this lower bound is given in the following proposition.

\begin{prop} \label{thmrandlower}
  Let $k \geq 1$, $t \geq 2$.  For all constants $c_1,c_2 > 0$, there exists a constant $d > 0$
  depending only on $k$ and $c_1,c_2$ such that if $c_1\log^2 t \leq m \leq c_2 2^t$ then
  $r_k(K_{2,t};K_m) \geq d\frac{m^2 t}{\log^2(mt)}$.
\end{prop}

\reflinearlower and \refrandlower take care of the left two columns in Table~\ref{ourtable}.
\reflinearlower works in both columns and most likely \refrandlower also works in both columns,
although we do not prove that since \reflinearlower is better when $m \ll \log^2 t$.  What about the
range $m \gg 2^t$?  As mentioned, an extension of \refrandlower using the random graph $G(n,p)$
gives a lower bound of $c_t m^{2-2/t}$ for some constant $c_t$ depending on $t$. When $t$ is
constant, Alon and R\"odl's~\cite{mr-alon05} result from Table~\ref{alontable} shows lower bounds of
$m^2/\log^4 m$ and $m^2/\log^2 m$ depending on $k$.  If $t$ is not fixed but still much smaller than
$m$, we can prove the following precise lower bounds.  This is our main theorem.

\begin{thm} \label{alllowerthm}
  Let $t \geq 2$ and $k \geq 3$.  There exists a constant $\rho > 0$
  depending only on $k$ such that the following holds.
  \begin{enumerate}
    \item[(i)] If $m \geq 128 \log^2 t$, then $r(K_{2,t},K_{2,t},K_m) \geq \rho \frac{m^2t}{\log^4(mt)}$.
    \item[(ii)] If $m \geq 16 k \log t$, then $r_k(K_{2,t};K_m) \geq \rho \frac{m^2t}{\log^2(mt)}$.
  \end{enumerate}
\end{thm}

The construction in the above theorem works for $k \geq 2$ and (roughly) the rightmost two columns
in Table~\ref{ourtable}.  When $k = 2$, it is slightly worse than the random graph construction from
\refrandlower and matches it when $k \geq 3$.  But it has the advantage over the random graph of
working in the rightmost column of Table~\ref{ourtable}, where $m$ is much larger than $t$.  Also,
the construction only works for $k \geq 2$, which is the reason for the missing lower bound in the
upper right cell of Table~\ref{ourtable}.

This construction is an algebraic graph construction using finite fields and is similar to a
construction by Lazebnik and Mubayi~\cite{ffc-lazebnik02}, which in turn was based on constructions
of Axenovich, F\"uredi, and Mubayi~\cite{ffc-axenovich00} and F\"uredi~\cite{ffc-furedi96}.  A
theorem of Alon and R\"odl~\cite{mr-alon05} which relates the second largest eigenvalue of a graph
with the number of the independent sets is then used to show the construction is a good choice
for a $K_{2,t}$-free graph with small independence number.  The properties of the construction are
stated in the following theorem.

\begin{thm}\label{primeconstr}
  For any prime power $q$ and any integer $t \geq 2$ such that $q \equiv 0 \pmod t$ or $q \equiv 1 \pmod t$,
  there exists a graph $G$ with the following properties:
  \begin{itemize}
    \setlength{\itemsep}{1pt} \setlength{\parskip}{0pt} \setlength{\parsep}{0pt}
    \item $G$ has $q(q-1)/t$ vertices,
    \item $G$ has no multiple edges but some vertices have loops,
    \item $G$ is regular of degree $q-1$ (loops contribute one to the degree),
    \item $G$ is $K_{2,t+1}$-free,
    \item the second largest eigenvalue of the adjacency matrix of $G$ is $\sqrt{q}$.
  \end{itemize}
\end{thm}

Several open problems remain: in Table~\ref{ourtable} are the upper or lower bounds correct?  The
upper and lower bounds are very close; we are fighting against a poly-log term.  But it would still
be interesting to know which bounds are correct.  One of the differences is a $\log^2 m$ versus a
$\log^2 (mt)$ in the denominator.  If $m$ is much larger than $t$ then $\log^2 m \sim \log^2(mt)$,
but in the left two columns the gap starts to widen.  As $m$ gets smaller relative to $t$, the
$m^2t/\log^2(mt)$ lower bound eventually becomes worse than a really simple $mt$ lower bound.

Other open problems include $r(K_{2,t},K_m)$ when $m$ is much larger than $t$ and 
$r_k(K_{s,t};K_m)$ when $s$ is larger than two.  Using ideas from the projective norm graphs, the
construction in Section~\ref{secconstruction} can be extended to use norms to forbid $K_{s,t}$ for
$s$ fixed, at the expense of more complexity in the proof of the spectrum.  Thus the remaining
problem on $r_k(K_{s,t};K_m)$ is to investigate when $s$, $t$, and $m$ are all going to infinity.  In
other words, how do the constants (implicit) in Table~\ref{ourtable} depend on $s$?  Comments about
these and other open problems are discussed in Section~\ref{secopenprob}.

\section{The Ramsey Numbers $r_k(K_{2,t};K_m)$} \label{secramsey} 
In this section we prove all the upper and lower bounds given in Table~\ref{ourtable}:
Proposition~\ref{ramupper} in Section~\ref{secramupper}, Proposition~\ref{thmrandlower} in
Section~\ref{secrandlower}, and Theorem~\ref{alllowerthm} in Section~\ref{secalglower}.

\subsection{An upper bound} \label{secramupper} 

In this section, we prove Proposition~\ref{ramupper}.  For two colors, the proposition was first
proved in the 1980s by Szemer\'edi but he never published a proof.  Caro, Rousseau, and
Zhang~\cite{ram-caro00} published a proof in 2000 and Jiang and Salerno~\cite{ram-jiang10} gave
another more general proof but still for two colors.  We use a slightly different (but closely
related) proof technique inspired by Alon and R\"odl~\cite{mr-alon05} to extend the upper bound to
three or more colors.  First, we need the following two theorems.
If $F$ is a graph and $n$ is an integer, define $ex(n,F)$ to be the maximum number of edges in an
$n$-vertex graph which does not contain $F$ as a subgraph.

\begin{thm} \label{thm:kovarisosturan}
  (K\"ovari, S\'os, Tur\'{a}n \cite{tbi-kovari54}) For $2 \leq t \leq n$, $ex(n,K_{2,t}) \leq
  \frac{1}{2} \sqrt{t-1} n^{3/2} + \frac{n}{2} \leq \sqrt{t} n^{3/2}$.
\end{thm}

The following theorem is a corollary of the famous result of Ajtai, Koml\'os, and
Szemer\'edi~\cite{ram-ajtai80} on $r(K_3,K_m)$ (see also \cite[Lemma 12.16]{randg-bollobasbook}.)

\begin{thm} \label{thm:sparsenbrhood}
  There exists an absolute constant $c$ such that the following holds.  Let $G$ be an $n$-vertex
  graph with average degree $d$ and let $s$ be the number of triangles in $G$.  Then
  \begin{align*}
    \alpha(G) \geq \frac{cn}{d} \left( \log d - \frac{1}{2} \log \left( \frac{s}{n} \right) \right).
  \end{align*}
\end{thm}

We will apply this theorem in a graph where we can bound the average degree and know a bound on the
number of edges in any neighborhood; using standard tricks the theorem can be changed to use average
degree.

\begin{cor} \label{cor:everynbr}
  There exists an absolute constant $c$ such that the following holds.  Let $G$ be an $n$-vertex
  graph with average degree at most $d$, where for every vertex $v \in V(G)$, every $2d$-subset of
  $N(v)$ spans at most $d^2/f$ edges.  Then the independence number of $G$ is at least $\frac{cn
  \log f}{d}$.
\end{cor}

\begin{proof} 
Let $H$ be the subgraph of $G$ formed by deleting all vertices with degree bigger than $2d$.  $H$
has at least half the vertices of $G$ since $G$ has average degree at most $d$; in addition $H$ has
maximum degree $2d$.  Also, $H$ has at most $s = nd^2/f$ triangles since each neighborhood of a
vertex in $H$ spans at most $d^2/f$ edges.  Thus Theorem~\ref{thm:sparsenbrhood} implies there
exists a constant $c$ so that
\begin{align*}
  \alpha(G) \geq \frac{cn}{d} \left( \log d - \frac{1}{2} \log \left( \frac{d^2}{f} \right) \right)
  = \frac{cn}{d} \left( \log d - \log \left( \frac{d}{\sqrt{f}} \right) \right)
  = \frac{cn}{2d} \log f.
\end{align*}
\end{proof} 

\begin{proof}[Proof of Proposition~\ref{ramupper}] 
Let $c_1$ be the constant from Corollary~\ref{cor:everynbr}; note that we can assume $c_1 \leq 1$.
Define $c_2 = \frac{256k^2}{c_1^2}$ and assume $n > \frac{c_2 m^2 t}{\log^2 m}$. Consider a
$(k+1)$-coloring of $E(K_n)$ and let $C_i$ be the graph whose edges are the $i$th color class for $i
= 1,\dots,k$.  Assume $C_i$ is $K_{2,t}$-free for all $1 \leq i \leq k$.  We will show that the
independence number of $C_1 \cup \dots \cup C_k$ is at least $m$, which will imply the
$(k+1)$-st color class contains a copy of $K_m$; i.e.  $r_k(K_{2,t};K_m) \leq \frac{c_2 m^2t}{\log^2
m}$.

Since $C_1,\dots,C_k$ are $K_{2,t}$-free, they each have at most $\sqrt{t} n^{3/2}$ edges by
Theorem~\ref{thm:kovarisosturan}. Let $G = C_1 \cup \dots \cup C_k$ so $\left| E(G) \right| \leq
k\sqrt{t} n^{3/2}$.  Let $d = 2k\sqrt{tn}$, so that $G$ has average degree at most $d$.  Consider
some vertex $v \in V(G)$ and let $A \subseteq N(v)$ with $|A| = 2d$.  Then $C_i[A]$ is
$K_{2,t}$-free for $1 \leq i \leq k$ so $\left|E(G[A])\right| \leq k \cdot ex(2d,K_{2,t}) \leq
4k\sqrt{t}d^{3/2}$.  To apply Corollary~\ref{cor:everynbr}, we need to solve the following for $f$:
\begin{align*}
  4k\sqrt{t}d^{3/2} = \frac{d^2}{f}.
\end{align*}
The solution is $f = \frac{1}{4k} \sqrt{d/t}$ so Corollary~\ref{cor:everynbr} implies $G$ contains
an independent set of size $\frac{c_1 n \log f}{d}$.  To complete the proof, we just need to show
this is at least $m$.  Use the definitions of $d = 2k\sqrt{tn}$ and $f = \frac{1}{4k} \sqrt{d/t}$ to
obtain
\begin{align*}
  \alpha(G) \geq \frac{c_1 n}{d} \log f 
  = \frac{c_1 n}{2k \sqrt{t n}} \log \left( \frac{1}{4k} \frac{\sqrt{2k} \sqrt[4]{tn}}{\sqrt{t}} \right) 
  = \frac{c_1}{2k} \sqrt{\frac{n}{t}} \log \left( \frac{1}{2\sqrt{2k}} \sqrt[4]{\frac{n}{t}} \right).
\end{align*}
Recall that we assumed $n > \frac{c_2 m^2 t}{\log^2 m}$, so
\begin{align*}
\alpha(G) \geq 
  \frac{c_1}{2k} \sqrt{\frac{c_2 m^2}{\log^2 m} } \log \left( \frac{1}{2\sqrt{2k}} \sqrt[4]{\frac{c_2
  m^2}{\log^2 m}} \right).
\end{align*}
Use that $c_2 = \frac{256k^2}{c_1^2}$ and simplify to obtain
\begin{align*}
\alpha(G) \geq 
  \frac{8m}{\log m} \log \left( \sqrt{\frac{\sqrt{c_2}}{8k} \cdot \frac{m}{\log m}} \right) = 
  \frac{4m}{\log m} \log \left( \frac{2}{c_1} \frac{m}{\log m} \right).
\end{align*}
Since $c_1 \leq 1$,
\begin{align*}
\alpha(G) \geq \frac{4m}{\log m} \log\left( \frac{m}{\log m} \right) 
        = \frac{4m}{\log m} \left( \log m - \log \log m \right) \geq m.
\end{align*}
The last inequality uses $\log m \geq \frac{4}{3} \log \log m$ which is true for $m \geq 3$.
\end{proof} 

\subsection{The Random Graph} \label{secrandlower} 

In this section, we prove \refrandlower by using the random graph $G(n,p)$.

\begin{lemma}\label{singleGnp}
  For all constants $c_1,c_2$, there exists a constant $c_3$ such that the following holds.  Given
  two integers $t$ and $m$ with $c_1\log^2 t \leq m \leq c_2 2^t$, let $n = c_3
  \frac{m^2t}{\log^2\left( mt \right)}$ and $p = \sqrt{\frac{t}{e^{8}n}}$.  Then with probability
  tending to $1$ as $m$ tends to infinity ($m \rightarrow \infty$ implies $t,n\rightarrow \infty$ as
  well), $G(n,p)$ is $K_{2,t}$-free and has independence number at most $m$.
\end{lemma}

\begin{proof} 
Let $c_3 = \min\{\frac{1}{c_2^2}, \frac{1}{400e^8}\}$.  The expected number of $K_{2,t}$s is upper bounded by
\begin{align} \label{eq:randnumK2t}
  n^2 \binom{n}{t} p^{2t} \leq n^2 \left( \frac{en}{t} \right)^t \left( \frac{t}{e^8n} \right)^t 
  = n^2 e^{-7t}.
\end{align}
We want this to go to zero as $m \rightarrow \infty$, so it suffices to show that $t$ is bigger than
roughly $\log n$.  Using the definition of $n$, upper bound $\log n$ by
\begin{align*}
  \log n = \log\left( c_3 \frac{m^2 t}{\log^2\left( mt \right)} \right) 
     \leq 2 \log m + \log t + \log c_3
\end{align*}
But since $m \leq c_2 2^t \leq c_2 e^t$,
\begin{align*}
\log n \leq 2 (\log c_2 + t) + \log t + \log c_3 
  \leq 2 t + \log t + 2 \log c_2 + \log c_3.
\end{align*}
Since $c_3 \leq \frac{1}{c_2^2}$, $2\log c_2 + \log c_3 \leq 0$.  Using that $\log t \leq t$, we obtain
$\log n \leq 3t$, which when combined with \eqref{eq:randnumK2t} shows the expected number of
$K_{2,t}s$ is upper bounded by
\begin{align*}
  n^2 e^{-7t} = e^{2\log n - 7t} \leq e^{-t}.
\end{align*}
Since $m \rightarrow \infty$ implies $t \rightarrow \infty$, the expected number of $K_{2,t}$s goes
to zero as $m \rightarrow \infty$.

Let $d = pn$.  When $d = o(n)$, the independence number of $G(n,p)$ is concentrated around
$\frac{2n}{d} \log d$.  More precisely, Frieze~\cite{randg-frieze90} (see
also~\cite{prob-alonspencer,randg-bollobasbook}) proved that for fixed $\epsilon > 0$ and $d =
o(n)$, with probability going to one as $n \rightarrow \infty$, the independence number of $G(n,p)$
is within $ \frac{\epsilon n}{d}$ of $\frac{2n}{d} (\log d - \log \log d - \log 2 + 1)$.  First,
note that since $c_1 \log^2 t \leq m$, $m^2t/\log^2(mt) \rightarrow \infty$ as $m \rightarrow
\infty$.  This implies $n/t \rightarrow \infty$ which implies $d = pn = o(n)$, so the result of
Frieze~\cite{randg-frieze90} can be applied.  Therefore, w.h.p.
\begin{align*}
  \alpha(G(n,p)) &< 10 \frac{2n}{pn} \log(pn) 
      = 20e^{4} \sqrt{\frac{n}{t}} \log\left(
  \sqrt{ \frac{nt}{e^8} } \right) \leq 10e^4 \sqrt{\frac{n}{t}} \log(nt).
\end{align*}
The next step is to show that when the definition of $n$ is inserted, the expression is at most $m$
showing w.h.p.\ the independence number of $G(n,p)$ is at most $m$.  The computations are very
similar to the end of the proof of Proposition~\ref{ramupper} in Section~\ref{secramupper}.
\begin{align*}
  \alpha(G(n,p)) 
  &< 10e^4 \sqrt{c_3} \frac{m}{\log(mt)} \log\left( \frac{c_3m^2 t^2}{\log^2(mt)} \right)
  \leq 20e^4 \sqrt{c_3} m \leq m.
\end{align*}

Therefore, as $m$ tends to infinity, the probability that $G(n,p)$ contains a copy of $K_{2,t}$ or has
independence number at least $m$ tends to zero, completing the proof.
\end{proof} 

\begin{proof}[Proof of \refrandlower] 
Color $E(K_n)$ by $k+1$ colors as follows: let the first color correspond to $G(n,p)$ with $p =
\sqrt{t/(e^8n)}$, do not assign any edges to colors $2, \dots, k$, and let the $(k+1)$st color be
the remaining edges (complement of the first color).  Lemma~\ref{singleGnp} shows w.h.p.\ the first
color is $K_{2,t}$-free (since $k$ is fixed) and the $(k+1)$st color has clique number at most $m$.
\end{proof} 

\subsection{An algebraic lower bound} \label{secalglower} 

In this subsection, we prove Theorem~\ref{alllowerthm}.
Our main tool is the following very general theorem from Alon and R\"odl~\cite{mr-alon05}.
Their idea is to take an $H$-free graph $G$ and construct $k$ graphs
$G_1,\ldots,G_k$ by taking $k$ random copies of $G$.  In other words,
fix some set $W$ of size $\left|V(G)\right|$ and let
$G_i$ be the graph obtained by a random bijection between $V(G)$ and $W$.
We now have a $k+1$ coloring of the edges of the complete graph on vertex set $W$: let the first $k$ colors be $G_1, \dots, G_k$
and let the $k+1$st color be the edges outside any $G_i$.
Alon and R\"odl's key insight is that if we know the second largest eigenvalue of $G$,
then $G$ is an expander graph which implies some knowledge about the independent
sets in $G$.  This is then used to bound the independence number of $G_1 \cup \dots \cup G_k$,
in other words obtain an estimate of $m$.

\begin{thm}\label{alonlower}
  (Alon and R\"odl, Theorem 2.1 and Lemma 3.1 from \cite{mr-alon05}) Let $G$ be an $n$-vertex,
  $H$-free, $d$-regular graph where $G$ has no multiple edges but some vertices have loops and let $k
  \geq 2$ be any integer.  Let $\lambda$ be the second largest eigenvalue in absolute value of the
  adjacency matrix of $G$.  If $m \geq \frac{2n}{d} \log n$ and 
  \begin{align} \label{eq:alonlower}
    \left( \frac{emd^2}{4\lambda n \log n}\right)^{\frac{2kn\log n}{d}} \left( \frac{2e\lambda n}{md} \right)^{km} \left(\frac{m}{n} \right)^{m(k-1)} < 1
  \end{align}
  then $r_k(H;K_m) > n$.
\end{thm}

A combination of Theorem~\ref{alonlower} and Theorem~\ref{primeconstr} plus the density of the prime
numbers proves Theorem~\ref{alllowerthm}.  To be able to apply Theorem~\ref{primeconstr}, we need to
find a prime power $q$ which is congruent to zero or one modulo $t$ and is in the required range.
Recall that we are targeting a bound of $\frac{m^2 t}{\log^4(mt)}$ or $\frac{m^2 t}{\log^2(mt)}$ and
the number of vertices from Theorem~\ref{primeconstr} is $q(q-1)/t$.  Given inputs $m$ and $t$, we
therefore want to find a prime power $q$ so that $q \equiv 0 \pmod t$ or $q \equiv 1 \pmod t$ and
$q(q-1)/t$ is near $\frac{m^2 t}{\log^{2s}(mt)}$ where $s$ is one or two.  This can be accomplished
using the Prime Number Theorem.

\begin{lemma} \label{primedensity}
  Fix integers $s, L \geq 1$.  There exists a constant $\delta > 0$ depending only on $s$ and $L$
  such that the following holds.  For every $t \geq 2$ and $m \geq 4^s L \log^{s}
  t$, either $\frac{\delta m^2 t}{L^2\log^{2s}(mt)} \leq 2$ or
  there is a prime power $q$ so that $q \equiv 1 \pmod{t}$ and
    \begin{align*}
       \delta \frac{m^2 t}{L^2 \log^{2s}(mt)} \leq \frac{q(q-1)}{t} \leq \frac{m^2 t}{L^2\log^{2s} (mt)}.
    \end{align*}
\end{lemma}

\begin{onlyarxiv}
The proof of this lemma is given in Appendix~\ref{sec:primedensity}.
\end{onlyarxiv}
\begin{notarxiv}
The proof of this lemma appears online.
\end{notarxiv}
Now a combination of Lemma~\ref{primedensity}, Theorem~\ref{primeconstr}, and
Theorem~\ref{alonlower} plus some computations proves \refktwolower.

\begin{proof}[Proof of \refktwolower] 
Suppose $t \geq 2$, $k = 2$, and $m \geq 128 \log^2 t$ are given.  Fix $s = 2$ and $L = 8$ so that
the conditions of Lemma~\ref{primedensity} are satisfied. Choose $q$ and $\delta$ according to
Lemma~\ref{primedensity}.  Note that if $\frac{\delta m^2t}{L^2\log^4(mt)} \leq 2$, then trivially
$r(K_{2,t},K_{2,t},K_m) \geq 2 \geq \frac{\delta m^2t}{L^2 \log^4(mt)}$.  Therefore, assume that
\begin{align} \label{eq:rangeofq}
   \delta \frac{m^2 t}{64 \log^4(mt)} \leq \frac{q(q-1)}{t} \leq \frac{m^2 t}{64\log^{4} (mt)}.
\end{align}
Let $G$ be the graph from Theorem~\ref{primeconstr}.  Then $d$ (the
average degree) is $q-1$, $\lambda$ (the second largest eigenvalue in absolute value) is $\sqrt{q}$,
and $n = q(q-1)/t$.

To apply Theorem~\ref{alonlower}, we need to show that $m \geq \frac{2n}{d} \log n$ and also show
$k$, $m$, $\lambda$, $n$, and $d$ satisfy the inequality \eqref{eq:alonlower}.  We break this into
two steps: first we show that $m \geq \frac{n}{d} \log^2 n \geq \frac{2n}{d} \log n$ using the
choice of $q$ from Lemma~\ref{primedensity}.  Next, we let $m' = \frac{n}{d} \log^2 n$ and check the
inequality \eqref{eq:alonlower} with $k$, $m'$, $\lambda$, $n$, and $d$.  This shows
$r_k(K_{2,t};K_{m'}) > n$, and since $m \geq m'$, this implies $r_k(K_{2,t};K_m) > n$.  Using that
$n = q(q-1)/t$, equation \eqref{eq:rangeofq} shows $n > \frac{1}{64} \delta m^2 t/\log^4(mt)$.  If
$\rho \leq \frac{\delta}{64}$, we have proved $r(K_{2,t},K_{2,t},K_m) \geq \rho m^2 t/\log^4(mt)$.
Also, note that we can assume $n > n_0$ for some constant $n_0$ by choosing $\rho = \frac{\delta}{64
n_0}$  (since then $n \leq n_0$ implies $\rho m^2 t/\log^4(mt) \leq 1$.)

\textbf{Step 1}  We want to show $m \geq \frac{n}{d} \log^2 n$.  Start with \eqref{eq:rangeofq}:
\begin{align}
  n &\leq \frac{m^2 t}{64\log^4 (mt)} \nonumber \\
  64n \log^4 (mt) &\leq m^2 t. \label{inequalitym}
\end{align}
Take the log of both sides, to obtain
\begin{align*}
  \log 64 + \log n + \log \log^4(mt) &\leq 2\log m + \log t \leq 2 \log(mt) \\
  \log n &\leq 2 \log(mt).
\end{align*}
Combining this with \eqref{inequalitym} yields
\begin{align*}
  &n \log^4 n \leq 16 n \log^4(mt) \leq \frac{1}{4} m^2 t \\
  &\Rightarrow \quad m^2 \geq \frac{4n}{t} \log^4 n \\
  &\Rightarrow \quad m \geq 2 \sqrt{ \frac{n}{t}} \log^2 n = \frac{2\sqrt{q(q-1)}}{t} \log^2 n \geq
  \frac{q}{t} \log^2 n = \frac{n}{d} \log^2 n.
\end{align*}
\textbf{Step 2} Let $m' = \frac{n}{d}\log^2 n$.  We need to verify that
\begin{align*}
  \left( \frac{em'd^2}{4\lambda n \log n}\right)^{\frac{2kn\log n}{d}} \left( \frac{2e\lambda
  n}{m'd} \right)^{km'} \left(\frac{m'}{n} \right)^{m'(k-1)} < 1.
\end{align*}
Substitute in $k = 2$ and $m' = \frac{n}{d}\log^2 n$ in the exponent of the LHS and then take the $m'$th-root
to obtain
\begin{align*}
 \Lambda &:= 
      \left( \frac{em'd^2}{4\lambda n \log n}\right)^{\frac{4}{\log n}}
      \left( \frac{2e\lambda n}{m'd} \right)^{2} \left(\frac{m'}{n} \right). \\
\end{align*}
We must show $\Lambda < 1$.
Substitute in $m' = \frac{n}{d}\log^2 n$ and simplify to obtain
\begin{align}
 \Lambda &= 
      \left( \frac{ed\log n}{4\lambda}\right)^{\frac{4}{\log n}}
      \left( \frac{4e^2\lambda^2}{\log^4 n} \right) \left(\frac{\log^2 n}{d} \right)
  =
      \left( \frac{e^4d^4\log^4 n}{256\lambda^4}\right)^{\frac{1}{\log n}}
      \left( \frac{4e^2\lambda^2}{d\log^2 n} \right). \label{eq:con1}
\end{align}
Now $d = q-1$, $\lambda = \sqrt{q}$, and $n = q(q-1)/t$ so $\lambda^2/d = q/(q-1) \leq 2$ and
\begin{align*}
  \frac{d^4}{\lambda^4} &= \frac{(q-1)^4}{q^2} < q(q-1) = nt < n^2.
\end{align*}
Insert these inequalities into \eqref{eq:con1} to obtain
\begin{align*}
 \Lambda &<
      \left( \frac{e^4 n^2 \log^4 n}{256}\right)^{\frac{1}{\log n}}
      \left( \frac{8e^2}{\log^2 n} \right).
\end{align*}
Since $n^2 = e^{2\log n}$ raised to the power $1/\log n$ is a constant, when $n$ gets big the above
expression drops below $1$ (as mentioned above, we can assume $n > n_0$.) Therefore,
Theorem~\ref{alonlower} implies that $r(K_{2,t},K_{2,t},K_{m'}) > n$.  In Step 1, we showed that $m
\geq m'$ so $r(K_{2,t},K_{2,t},K_m) > n$.  Since $n = q(q-1)/t$, equation \eqref{eq:rangeofq} shows
that $n > \frac{1}{64} \delta \frac{m^2t}{\log^4 mt}$, completing the proof.
\end{proof} 

\begin{proof}[Proof sketch of \refkthreelower] 
Given $m$, $t$, and $k \geq 3$, fix $s = 1$ (instead of $2$) and $L = 4k$ and choose $q$ and
$\delta$ according to Lemma~\ref{primedensity}.  The proof is mostly the same as the above proof,
except we choose $m' = 2k\frac{n}{d} \log n$ (the difference is that the log is not squared plus now
there is a $2k$ out front.)  The proof then proceeds in two steps: show that $m \geq 2k\frac{n}{d}
\log n = m'$ and then show that $k$, $m'$, $\lambda$, $n$, and $d$ satisfy the inequality
\eqref{eq:alonlower}.  Showing $m \geq m'$ is almost identical to Step 1 in the previous proof.
Showing $k$, $m'$, $\lambda$, $n$, and $d$ satisfy inequality \eqref{eq:alonlower} in Theorem~\ref{alonlower}
is tedious; the details
\begin{onlyarxiv}
are in Appendix~\ref{appendixcalcs}.
\end{onlyarxiv}
\begin{notarxiv}
are online.
\end{notarxiv}
\end{proof} 

\section{An algebraic $K_{2,t+1}$-free construction} \label{secconstruction} 

To prove Theorem~\ref{primeconstr}, we construct two different graphs for the two cases: one
graph $G^+$ for $q \equiv 0 \pmod t$ and one graph $G^{\times}$ for $q \equiv 1 \pmod{t}$.
The two graphs are closely related;
they are built from finite fields.  Fix a prime $p$ and an integer $a$, and let $q = p^a$.
Let $\mathbb{F}_q$ be the finite field of order $q$ and
let $\mathbb{F}^*_q$ be the finite field of order $q$ without the zero element.  

When $q \equiv 0 \pmod t$, let $H$ be an additive subgroup of $\mathbb{F}_q$ of order $t$.
Such a subgroup exists since $t$ divides $q$ so $t = p^b$ for some $b\leq a$.
Define a graph $G^+$ as follows. Let $V(G^+) = \left( \mathbb{F}_q/H \right) \times \mathbb{F}^*_q$.
We will write elements of $\mathbb{F}_q/H$ as $\bar{a}$, where $\bar{a}$ as the additive coset
of $H$ generated by $a$.  That is, $\bar{a} = \left\{ h + a : h \in H \right\}$.
For $\bar{a}, \bar{b} \in \mathbb{F}_q/H$ and $x,y \in \mathbb{F}^*_q$, make
$(\bar{a},x)$ adjacent to $(\bar{b},y)$ if $xy \in \overline{a+b}$.  Since $H$ is a normal subgroup
the coset $\overline{a+b}$ is well-defined, so by $xy \in \overline{a+b}$ we mean there exists some
$h \in H$ such that $xy = h + a + b$.
\footnote{In finite fields, additive subgroups of a given order are isomorphic as groups.  Each
element of $\mathbb{F}_q$ has additive order the characteristic, so $H$ decomposes into
$p^{b-1}$ orbits of size $p$ and one can obtain a group isomorphism by mapping orbits to orbits.
Therefore, $G^+$ is uniquely defined up to isomorphism.}

When $q \equiv 1 \pmod{t}$, let $H$ be a multiplicative subgroup of $\mathbb{F}^*_q$ of order $t$.
Such a subgroup exists since $t$ divides the order of $\mathbb{F}^*_q$ and $\mathbb{F}^*_q$ is a
cyclic multiplicative group.  Define a graph $G^{\times}$ as follows.  Let $V(G^{\times}) = \left(
\mathbb{F}^*_q/H \right) \times \mathbb{F}_q$.  For $\bar{a}, \bar{b} \in \mathbb{F}^*_q/H$ and $x,y
\in \mathbb{F}_q$, make $(\bar{a},x)$ adjacent to $(\bar{b},y)$ if $x+y \in \overline{ab}$.
\footnote{In finite fields, multiplicative subgroups of a given order are isomorphic as groups since
$\mathbb{F}^*_q$ is cyclic.  Therefore, $G^\times$ is uniquely defined up to isomorphism.} 

\subsection{Simple properties of $G^+$ and $G^{\times}$} 

\begin{lemma} \label{gisregular}
  $G^+$ and $G^\times$ are regular of degree $q-1$.
\end{lemma}

\begin{proof} 
First, consider $G^+$.  Fix some vertex $(\bar{a},x) \in V(G^+)$ and pick $y \in \mathbb{F}_q^*$
($q-1$ choices.) The element $xy$ is now in some coset $\bar{c}$.  Since the cosets form a group,
the coset $\overline{c-a}$ is well-defined.  Thus $(\bar{a},x)$ is adjacent to $(\overline{d},y)$ in
$G^+$ if and only if $\overline{d} = \overline{xy-a}$.

Now consider $G^\times$.  Fix some vertex $(\bar{a},x) \in V(G^\times)$ and pick $y \in
\mathbb{F}_q$.  If $x \neq -y$, then there is a coset $\bar{c}$ containing $x+y$.  Since the cosets
form a group, the coset $\overline{ca^{-1}}$ is well defined.  If $x = -y$, then there
is no coset which contains zero.  Thus $(\bar{a},x)$ is adjacent to $(\overline{d},y)$ if and only
if $x \neq -y$ and $\overline{d} = \overline{(x+y)a^{-1}}$.  Therefore $(\bar{a},x)$ is adjacent to $q-1$
vertices, since there are $q-1$ choices for $y \in \mathbb{F}_q$ with $x \neq -y$.
\end{proof} 

\begin{lemma}
  The common neighborhood of any two vertices in $G^+$ has size exactly $t$.
\end{lemma}

\begin{proof} 
The proof is similar to the proofs given in \cite{ffc-furedi96,ffc-lazebnik02}.
Fix $\bar{a},\bar{b} \in \mathbb{F}_q/H$
and $x,y \in \mathbb{F}^*_q$ and consider the common neighborhood of the vertices $(\bar{a},x)$ and $(\bar{b},y)$.
A vertex $(\bar{c},z)$ will be adjacent to both of $(\bar{a},x)$ and $(\bar{b},y)$ if
\begin{align*}
  xz &\in \overline{a+c} \\
  yz &\in \overline{b+c}.
\end{align*}
In other words, there exists some $h_1,h_2 \in H$ such that
\begin{align*}
  xz &= a + c + h_1 \\
  yz &= b + c + h_2.
\end{align*}
So fix $h_1,h_2 \in H$ and count how many choices there are for $c$ and $z$ so that $(\bar{c},z)$ is
adjacent to both $(\bar{a},x)$ and $(\bar{b},y)$ using $h_1$ and $h_2$.  We show there is a unique
$c$ and $z$.  Say we had $c,c',z,z'$ such that
\begin{align}
  xz &= a + c + h_1 \label{eq:Gp1} \\
  yz &= b + c + h_2 \label{eq:Gp2} \\
  xz' &= a + c' + h_1 \label{eq:Gp3} \\
  yz' &= b + c' + h_2 \label{eq:Gp4}.
\end{align}
Add \eqref{eq:Gp1} to \eqref{eq:Gp4}; this equals \eqref{eq:Gp2} plus \eqref{eq:Gp3}.
\begin{align}
  xz + yz' &= a + b + c + c'+ h_1 + h_2 = yz + xz' \nonumber \\
  (x-y)(z-z') &= 0 \label{eq:Gp5}.
\end{align}
If $x = y$, then subtracting \eqref{eq:Gp1} from \eqref{eq:Gp2} gives $a-b \in H$ which means $\bar{a} = \bar{b}$.  But now
$(\bar{a},x)$ and $(\bar{b},y)$ are the same vertex.  Thus \eqref{eq:Gp5} implies $z = z'$.
Then subtracting \eqref{eq:Gp1} and \eqref{eq:Gp3} we get $c = c'$, showing there is a unique $c,z$ such
that $(\bar{c},z)$ is adjacent to both $(\bar{a},x)$ and $(\bar{b},y)$ using $h_1,h_2$.  (Note that not only is there
a unique $(\bar{c},z)$, but the choice of the representative $c$ for the coset $\bar{c}$ is unique.)

There are now $t^2$ choices for $h_1$ and $h_2$ and each provides a unique $c,z$.  But each coset $\bar{c}$
has $t$ elements so there are exactly $t^2/t = t$ common neighbors of $(\bar{a},x)$ and $(\bar{b},y)$.
\end{proof} 

\begin{lemma}
  The common neighborhood of any two vertices in $G^{\times}$ has size exactly $t$.
\end{lemma}

\begin{proof} 
Fix $\bar{a},\bar{b} \in \mathbb{F}^*_q/H$ and $x,y \in \mathbb{F}_q$ and consider the common
neighborhood of the vertices $(\bar{a},x)$ and $(\bar{b},y)$.  A vertex $(\bar{c},z)$ will be
adjacent to both $(\bar{a},x)$ and $(\bar{b},y)$ if
\begin{align*}
  x + z &\in \overline{ac} \\
  y + z &\in \overline{bc}.
\end{align*}
In other words, there exists some $h_1, h_2 \in H$ such that
\begin{align*}
  x + z &= h_1ac \\
  y + z &= h_2bc.
\end{align*}
So fix some $h_1, h_2 \in H$ and count how many choices there are for $c$ and $z$ so that
$(\bar{c},z)$ is adjacent to both $(\bar{a},x)$ and $(\bar{b},y)$ using $h_1$ and $h_2$.  We show
there is a unique such $c$ and $z$.  Say there existed $c,c',z,z'$ such that
\begin{align}
  x + z &= h_1ac \label{eq:Gt1} \\
  y + z &= h_2bc \label{eq:Gt2} \\
  x + z' &= h_1ac' \label{eq:Gt3} \\
  y + z' &= h_2bc'. \label{eq:Gt4}
\end{align}
Multiply \eqref{eq:Gt1} by \eqref{eq:Gt4}, which equals \eqref{eq:Gt2} times \eqref{eq:Gt3}.
\begin{align}
  (x+z)(y+z') &= h_1h_2abcc' = (y+z)(x+z')  \nonumber \\
  xy + xz' + yz + zz' &= xy + yz' + xz + zz' \nonumber \\
  xz' + yz &= xz + yz' \nonumber \\
  (x - y)(z' - z) &= 0 \label{eq:Gt5}
\end{align}
If $x = y$, then \eqref{eq:Gt1} and \eqref{eq:Gt2} show
\begin{align*}
  h_1 a c &= x + z = y + z = h_2 b c \\
  ab^{-1} &= h_1^{-1}h_2 \in H
\end{align*}
which shows $\bar{a} = \bar{b}$.  But now $(\bar{a},x)$ and $(\bar{b},y)$ are the same vertex.
Thus \eqref{eq:Gt5} implies $z = z'$.  But now \eqref{eq:Gt1} and \eqref{eq:Gt3} show $c = c'$.

Thus for every choice of $h_1, h_2 \in H$ there is a unique $c,z$ such that $(\bar{c},z)$ is
adjacent to both $(\bar{a},x)$ and $(\bar{b},y)$ using $h_1,h_2$.  Note that not only is there a
unique $(\bar{c},z)$, but the choice of the representative $c$ for the coset $\bar{c}$ is unique.
There are now $t^2$ choices for $h_1$ and $h_2$ and each provides a unique $c,z$.  But each coset
$\bar{c}$ has $t$ elements so there are exactly $t^2/t = t$ common neighbors of $(\bar{a},x)$ and
$(\bar{b},y)$.
\end{proof} 

\subsection{The Spectrum of $G^+$ and $G^{\times}$} \label{sec:spectrum} 

\begin{lemma} \label{eigenvaluesGplus}
  The eigenvalues of $G^+$ are $q-1$, $\pm \sqrt{q}$, $\pm 1$, and $0$.  If $p$ is an
  odd prime, they have the following multiplicities: $q-1$ has multiplicity $1$, $\sqrt{q}$ and
  $-\sqrt{q}$ each have multiplicity $\frac{1}{2}(q/t-1)(q-2)$, $1$ and $-1$ both have multiplicity
  $\frac{1}{2}\left(q/t - 1\right)$, and $0$ has multiplicity $q-2$.
\end{lemma}

\begin{lemma} \label{eigenvaluesGtimes}
  The eigenvalues of $G^\times$ are $q-1$, $\pm \sqrt{q}$, $\pm 1$, and $0$.  If $p$ is an
  odd prime, they have the following multiplicities: $q-1$ has multiplicity $1$, $\sqrt{q}$ and
  $-\sqrt{q}$ each have multiplicity $\frac{1}{2}((q-1)/t-1)(q-1)$, $1$ and $-1$ both have multiplicity
  $\frac{1}{2}\left(q - 1\right)$, and $0$ has multiplicity $(q-1)/t - 1$.
\end{lemma}

The proof of these lemmas are similar to proofs by Alon and R\"{o}dl~\cite[Lemma 3.6]{mr-alon05} and
Szab\'o~\cite{ee-szabo03}.  In addition, the two proofs given below are almost the same but there
are several subtle issues involving the fact that $G^+$ and $G^\times$ switch between $\mathbb{F}_q$
and $\mathbb{F}^*_q$. There are small but crucial differences in how the proofs below handle the
zero element.  Therefore, we give both proofs and caution the reader to pay attention to how the
zero element is handled when reading the proofs.

\begin{proof}[Proof of Lemma~\ref{eigenvaluesGplus}] 
Let $M$ be the adjacency matrix of $G^+$.  Let $\chi$ be an arbitrary additive character of $\mathbb{F}_q/H$
and let $\phi$ be an arbitrary multiplicative character of $\mathbb{F}^*_q$.  This means that
\begin{align*}
  \chi : \mathbb{F}_q/H \rightarrow \mathbb{C} \quad \quad
  \phi : \mathbb{F}^*_q \rightarrow \mathbb{C}
\end{align*}
where $\chi$ is an additive group homomorphism (if $\bar{a}, \bar{b}$ are cosets in $\mathbb{F}_q/H$
then $\chi(\bar{a}+\bar{b}) = \chi(\bar{a})\chi(\bar{b})$, $\chi(\bar{0}) = 1$, and $\chi(-\bar{a})
= \chi(\bar{a})^{-1}$) and $\phi$ is a multiplicative group homomorphism (if $a,b \in
\mathbb{F}^*/q$ then $\phi(ab) = \phi(a)\phi(b)$, $\phi(1) = 1$, $\phi(a^{-1}) = \phi(a)^{-1}$.)
Note that since $\phi(1) = 1$ and $x^q = 1$ for any $x \in \mathbb{F}^*_q$, $\phi(x)$ must be a root
of unity in $\mathbb{C}$.  Thus $\phi(x^{-1}) = \phi(x)^{-1} = \overline{\phi(x)}$ where
$\overline{\phi(x)}$ is the complex conjugate of $\phi(x)$.  Similarly, $\chi(-\bar{a}) =
\overline{\chi(\bar{a})}$, the complex conjugate of $\chi$ applied to the coset $\bar{a}$.

Let $\left<\chi,\phi\right>$ denote the column vector whose coordinates are labeled by the elements
of $V(G^+)$ and whose entry
at the coordinate $(\bar{a},x)$ is $\chi(\bar{a})\phi(x)$.
We now show that $\left<\chi,\phi\right>$ is an eigenvector of $M$ and compute its eigenvalue.  The following
expression is the entry of the vector $M\left<\chi,\phi\right>$ at the coordinate $(\bar{a},x)$.
$$\sum_{\substack{ (\bar{b},y) \text{ is a vertex } \\ (\bar{a},x) \leftrightarrow (\bar{b},y)}} \chi(\bar{b})\phi(y) =
\sum_{\substack{ \bar{b} \in \mathbb{F}_q/H \\ y \in \mathbb{F}^*_q \\ xy \in \overline{a+b}} } \chi(\bar{b})\phi(y)$$

First, we make two changes of variables in this sum. The first change is to switch $\bar{b}$ to $\bar{c}$
by the transformation $\bar{c} = \overline{a + b} = \bar{a} + \bar{b}$.
$$\sum_{\substack{ \bar{c} \in \mathbb{F}_q/H \\ y \in \mathbb{F}^*_q \\ xy \in \bar{c}} } \chi(\overline{c-a})\phi(y)$$
Next, switch $y$ to $z$ by the transformation $z=xy$.
$$\sum_{\substack{ \bar{c} \in \mathbb{F}_q/H \\ z \in \mathbb{F}^*_q \\ z \in \bar{c}} } \chi(\overline{c-a})\phi\left(\frac{z}{x}\right).$$
Using that $\chi$ and $\phi$ are characters (homomorphisms), this transforms to
\begin{align*}
  \left(\chi(\bar{a}) \phi(x) \right)^{-1}
  \sum_{\substack{ \bar{c} \in \mathbb{F}_q/H \\ z \in \mathbb{F}^*_q \\ z \in \bar{c}} } \chi(\bar{c})\phi(z)
  = \overline{\chi(\bar{a})} \overline{\phi(x)}
  \sum_{\left\{ (\bar{c},z) : \bar{c} \in \mathbb{F}_q/H, z \in \mathbb{F}^*_q, z \in \bar{c} \right\} }
  \chi(\bar{c}) \phi(z)
\end{align*}
There is an obvious bijection between the set
$\left\{ (\bar{c},z) : \bar{c} \in \mathbb{F}_q/H, z \in \mathbb{F}^*_q, z \in \bar{c} \right\}$ and the
set $\left\{ z : z \in \mathbb{F}^*_q \right\}$, since once $z$ is picked, there is a unique coset containing $z$.
Thus the above sum can be simplified to
$$ \overline{\chi(\bar{a})} \overline{\phi(x)} \sum_{z \in \mathbb{F}^*_q} \chi(\bar{z})\phi(z).$$

Define
$\Gamma_{\chi,\phi} = \sum_{z \in \mathbb{F}^*_q} \chi(\bar{z})\phi(z)$
so that $\Gamma_{\chi,\phi}$ is some constant depending only on $\chi$ and $\phi$.
Then the vector $M\left<\chi,\phi\right>$ is $\Gamma_{\chi,\phi} \left<\overline{\chi},\overline{\phi}\right>$.
Thus $M^2 \left<\chi,\phi\right> = \Gamma_{\chi,\phi}\overline{\Gamma_{\chi,\phi}} \left<\chi,\phi\right>$,
so $\Gamma_{\chi,\phi}\overline{\Gamma_{\chi,\phi}}$ is an eigenvalue of $M^2$.

\begin{lemma} \label{lem:charfacts}
  Let $A$ be a finite group.  There are $|A|$ characters of $A$ and if $\tau : A \rightarrow \mathbb{C}$ is
  a non-principal character then $\sum_{a \in A} \tau(a) = 0$.
\end{lemma}

The above lemma shows there are $|\mathbb{F}_q/H| \cdot |\mathbb{F}^*_q| = q(q-1)/t = |V(G^+)|$
vectors $\left<\chi,\phi\right>$.  Secondly, the lemma shows $\left<\chi,\phi\right>$ is orthogonal
to $\left<\chi',\phi'\right>$ if $\chi \neq \chi'$ or $\phi \neq \phi'$  (the dot
product of $\left<\chi,\phi\right>$ with $\left<\chi',\phi'\right>$ is a sum which can be rearranged
to apply Lemma~\ref{lem:charfacts}.)

Since $\left\{ \left<\chi,\phi\right> : \chi,\phi \text{ characters } \right\}$ is a linearly
independent set of $|V(G^+)|$ eigenvectors of $M^2$ and $M^2$ has $|V(G^+)|$ columns,  all eigenvalues of
$M^2$ are of the form $\Gamma_{\chi,\phi}\overline{\Gamma_{\chi,\phi}}$.  The eigenvalues of $M^2$ are the squares of the
eigenvalues of $M$.  Since $M$ is symmetric, these eigenvalues are real so all eigenvalues of $M$
are of the form $\pm \left| \Gamma_{\chi,\phi} \right|$.

When $\chi$ and $\phi$ are principal characters of their respective groups (this means $\chi$ and $\phi$ map everything to $1$), the corresponding
eigenvalue is $q-1$ since there are $q-1$ terms in the sum defining $\Gamma_{\chi,\phi}$.  This eigenvalue has multiplicity one.
When $\chi$ is principal but $\phi$ is not principal, the eigenvalues are
$$\Gamma_{\chi,\phi} = \sum_{z \in \mathbb{F}^*_q} \phi(z) = 0.$$
There are $q-1$ possible characters $\phi$, but one of them is principal so $0$ will have multiplicity $q-2$
as an eigenvalue.
When $\phi$ is principal but $\chi$ is not, we obtain
$$\Gamma_{\chi,\phi} = \sum_{z \in \mathbb{F}^*_q} \chi(\bar{z}) = t \sum_{\bar{z} \in \mathbb{F}_q/H} \chi(\bar{z})  - \chi(\bar{0})
= -\chi(\bar{0}) = -1.$$
($\chi(\bar{0})$ is subtracted since the sum over $\mathbb{F}^*_q$ will have $t = \left| H \right|$ terms for each
coset, except the zero coset will only appear $t-1$ times.)  Thus the eignevalues when $\phi$ is principal
and $\chi$ is not are $\pm 1$.  For the multiplicities, there are $q/t - 1$ non-principal characters $\chi$.
They come in pairs, since if $\chi$ is a character, the complex conjugate 
$\overline{\chi}$ is a character as well.  Also, note that $\left<\chi,\phi\right> + \left<\bar{\chi},\phi\right>$ has eigenvalue $1$
and $\left<\chi,\phi\right> - \left<\bar{\chi},\phi\right>$ has eigenvalue $-1$ (when $\phi$
is principal.) Thus if $p$ is an odd prime, $1$ and $-1$ will
each have multiplicity $\frac{1}{2} (q/t - 1)$.  

When neither $\chi$ nor $\phi$ is a principal character, we apply a theorem on Gaussian sums of characters.

\begin{thm} \label{gausssums}
  If $\chi'$ and $\phi$ are additive and multiplicative non-principal characters of $\mathbb{F}_q$ and
  $\mathbb{F}^*_q$ respectively, then
  $\left| \sum_{x \in \mathbb{F}^*_q} \chi'(x) \phi(x) \right| = \sqrt{q}$.
\end{thm}

While we can't apply this theorem directly since $\chi$ is not a character on $\mathbb{F}_q$,
define a new additive character $\chi'$ on $\mathbb{F}_q$ as follows:
for $x \in \mathbb{F}_q$ let $\chi'(x) = \chi(\bar{x})$.  This is an additive character because
$\chi'(0) = \chi(\bar{0}) = 1$, $\chi'(x+y) = \chi(\overline{x+y}) = \chi(\bar{x} + \bar{y}) = \chi(\bar{x})\chi(\bar{y}) = \chi'(x)\chi'(y)$,
and $\chi'(-x) = \chi(\overline{-x}) = \chi(\bar{x})^{-1} = \chi'(x)^{-1}$.  We can now rewrite $\Gamma_{\chi,\phi}$ as
$$\Gamma_{\chi,\phi} = \sum_{z \in \mathbb{F}^*_q} \chi'(x) \phi(x).$$
Theorem~\ref{gausssums} shows that when $\chi$ and $\phi$ are both non-principal, the corresponding
eigenvalue is $\pm \sqrt{q}$.
\end{proof} 

\begin{proof}[Proof of Lemma~\ref{eigenvaluesGtimes}] 
Let $M$ be the adjacency matrix of $G^\times$.  Let $\chi$ be an arbitrary multiplicative character
of $\mathbb{F}^*_q/H$ and let $\phi$ be an arbitrary additive character of $\mathbb{F}_q$.

Let $\left<\chi,\phi\right>$ denote the column vector whose coordinates are labeled by the elements
of $V(G^{\times})$ and whose entry at the coordinate $(\bar{a},x)$ is $\chi(\bar{a})\phi(x)$.  We now show
that $\left<\chi,\phi\right>$ is an eigenvector of $M$ and compute its eigenvalue.  The following
expression is the entry of the vector $M\left<\chi,\phi\right>$ at the coordinate $(\bar{a},x)$.
$$\sum_{\substack{ (\bar{b},y) \text{ is a vertex } \\ (\bar{a},x) \leftrightarrow (\bar{b},y)}}
\chi(\bar{b})\phi(y) = \sum_{\substack{ \bar{b} \in \mathbb{F}^*_q/H \\ y \in \mathbb{F}_q \\ x+y
\in \overline{ab}} } \chi(\bar{b})\phi(y)$$

First, we make two changes of variables in this sum. The first change is to switch $\bar{b}$ to $\bar{c}$
by the transformation $\bar{c} = \overline{ab} = \bar{a} \cdot \bar{b}$.
$$\sum_{\substack{ \bar{c} \in \mathbb{F}^*_q/H \\ y \in \mathbb{F}_q \\ x+y \in \bar{c}} }
\chi(\overline{ca^{-1}})\phi(y)$$
Next, switch $y$ to $z$ by the transformation $z=x+y$.
$$\sum_{\substack{ \bar{c} \in \mathbb{F}^*_q/H \\ z \in \mathbb{F}_q \\ z \in \bar{c}} }
\chi(\overline{ca^{-1}})\phi\left(z - x\right).$$
Using that $\chi$ and $\phi$ are characters (homomorphisms), this transforms to
\begin{align*}
  \left(\chi(\bar{a}) \phi(x) \right)^{-1}
  \sum_{\substack{ \bar{c} \in \mathbb{F}^*_q/H \\ z \in \mathbb{F}_q \\ z \in \bar{c}} } \chi(\bar{c})\phi(z)
  = \overline{\chi(\bar{a})} \overline{\phi(x)}
  \sum_{\left\{ (\bar{c},z) : \bar{c} \in \mathbb{F}^*_q/H, z \in \mathbb{F}_q, z \in \bar{c} \right\} }
  \chi(\bar{c}) \phi(z)
\end{align*}
There is an obvious bijection between the set $\left\{ (\bar{c},z) : \bar{c} \in \mathbb{F}^*_q/H, z
\in \mathbb{F}_q, z \in \bar{c} \right\}$ and the set $\left\{ z : z \in \mathbb{F}^*_q \right\}$,
since once a non-zero $z$ is picked, there is a unique coset containing $z$.  (When $z = 0$, there is
no coset containing $z$.) Thus the above sum can
be simplified to $$ \overline{\chi(\bar{a})} \overline{\phi(x)} \sum_{z \in \mathbb{F}^*_q}
\chi(\bar{z})\phi(z).$$

Define $\Gamma_{\chi,\phi} = \sum_{z \in \mathbb{F}^*_q} \chi(\bar{z})\phi(z)$ so that
$\Gamma_{\chi,\phi}$ is some constant depending only on $\chi$ and $\phi$.  Then the vector
$M\left<\chi,\phi\right>$ is $\Gamma_{\chi,\phi} \left<\overline{\chi},\overline{\phi}\right>$.
Thus $M^2 \left<\chi,\phi\right> = \Gamma_{\chi,\phi}\overline{\Gamma_{\chi,\phi}} \left<\chi,\phi\right>$ so
$\Gamma_{\chi,\phi}\overline{\Gamma_{\chi,\phi}}$ is an eigenvalue of $M^2$.  Like the last proof, Lemma~\ref{lem:charfacts}
shows all eigenvalues of $M^2$ are of the form $\Gamma_{\chi,\phi}\overline{\Gamma_{\chi,\phi}}$ so
all eigenvalues of $M$ are of the form $\pm \left| \Gamma_{\chi,\phi} \right|$.

When $\chi$ and $\phi$ are principal characters of their respective groups, the corresponding
eigenvalue is $q-1$ since there are $q-1$ terms in the sum.  This eigenvalue has multiplicity one.
When $\phi$ is principal but $\chi$ is not principal, the eigenvalues are
$$\Gamma_{\chi,\phi} = \sum_{z \in \mathbb{F}^*_q} \chi(\bar{z}) = t \sum_{\bar{z} \in \mathbb{F}^*_q/H}
\chi(\bar{z}) = 0.$$
There are $(q-1)/t$ possible characters $\chi$, but one of them is principal so $0$ will have
multiplicity $(q-1)/t - 1$
as an eigenvalue.
When $\chi$ is principal but $\phi$ is not, we obtain
$$\Gamma_{\chi,\phi} = \sum_{z \in \mathbb{F}^*_q} \phi(z) 
= \sum_{z \in \mathbb{F}_q} \phi(z)  - \phi(0)
= -\phi(0) = -1.$$
Thus the eignevalues when $\chi$ is principal and $\phi$ is not are $\pm 1$.  
For the multiplicities, there are $q - 1$ non-principal characters $\phi$.
They come in pairs, since if $\phi$ is a character, the complex conjugate 
$\overline{\phi}$ is a character as well.  Also, note that $\left<\chi,\phi\right> +
\left<\chi,\bar{\phi}\right>$ has eigenvalue $1$
and $\left<\chi,\phi\right> - \left<\chi,\bar{\phi}\right>$ has eigenvalue $-1$ (when $\chi$
is principal.) Thus if $p$ is an odd prime, $1$ and $-1$ will
each have multiplicity $\frac{1}{2} (q - 1)$.

When neither $\chi$ or $\phi$ is a principal character, we apply Theorem~\ref{gausssums}.  While we
can't apply this theorem directly since $\chi$ is not a multiplicative character on
$\mathbb{F}^*_q$, define a new multiplicatve character $\chi'$ on $\mathbb{F}^*_q$ as follows: for
$x \in \mathbb{F}^*_q$ let $\chi'(x) = \chi(\bar{x})$.  This is a multiplicative character because
$\chi'(1) = \chi(\bar{1}) = 1$, $\chi'(xy) = \chi(\overline{xy}) = \chi(\bar{x} \cdot \bar{y}) =
\chi(\bar{x})\chi(\bar{y}) = \chi'(x)\chi'(y)$, and $\chi'(x^{-1}) = \chi(\overline{x^{-1}}) =
\chi(\bar{x})^{-1} = \chi'(x)^{-1}$.  We can now rewrite $\Gamma_{\chi,\phi}$ as
$$\Gamma_{\chi,\phi} = \sum_{z \in \mathbb{F}^*_q} \chi'(x) \phi(x).$$ Theorem~\ref{gausssums} shows
that when $\chi$ and $\phi$ are both non-principal, the corresponding eigenvalue is $\pm \sqrt{q}$.
\end{proof} 

\subsection{Independence number} 

In Table~\ref{ourtable}, there is no lower bound in the upper right cell; that is, when $m$ is much
larger than $t$ the only lower bound we know is the bound of $c_t m^{2-1/t}$ from the random graph.
What about using $G^+$ or $G^{\times}$ as the first color in a construction for the lower bound?  In
other words, what is the independence number of $G^+$ and $G^{\times}$?  This is related to the
conjecture that Paley Graphs are Ramsey Graphs (see \cite{pal-maistrelli06} and its references.)
While we aren't able to determine exactly the independence number, computation suggests that $G^+$
and $G^{\times}$ have independent sets of size roughly $\sqrt{n}$, where $n$ is the number of
vertices.  In particular, computation suggests the following conjecture for $G^+$.

\begin{conj}
  Let $G^+(q,t)$ be the graph constructed at the beginning of this section for the parameters $q$
  and $t$.  Recall that $G^+(q,t)$ has $q(q-1)/t$ vertices which is regular of degree $q-1$ so
  $G^+(2^a,2^{a-1})$ is an $n$-vertex graph where every degree is about $n/2$ and any pair of
  vertices have about $n/4$ common neighbors.
  For $a \geq 6$,
  \begin{align*}
    \alpha(G^+(2^a,2^{a-1})) &=
    \begin{cases}
      2^{a/2}  & \text{if $a$ is even} \\
      2^{(a-1)/2} + 1 & \text{if $a$ is odd}
    \end{cases} \\
    \alpha(G^+(p^2,p)) &= p^2-1 \quad \quad \text{ if $p$ is odd }
  \end{align*}
\end{conj}

Note that $\alpha(G^+(2^3,2^2)) = 4$ and $\alpha(G^+(2^4,2^3)) = 5$, which don't quite match the conjecture.
For $\alpha(G^+(2^a,2^{a-1})$, the conjecture is true for $a = 6, 7, 8, 9, 10$.
For $G^+(p^2,p)$, the conjectured value is $p^2-1$; we can prove a lower bound of $\frac{1}{2} p^2$.
First, we need the following simple lemma about finite fields and field extensions.

\begin{lemma} \label{frobenius}
  Let $p$ be a prime and let $x \in \mathbb{F}_{p^a}^*$ with $x$ a generator for the cyclic multiplicative group $\mathbb{F}_{p^a}^*$. Then
  \begin{align*}
    \left\{ 1,2,\ldots,p-1 \right\} = \left\{ x^{t(p^a-1)/(p-1)} : 0 \leq t < p-1  \right\}
  \end{align*}
\end{lemma}

\begin{proof} 
The Frobenius automorphism $\phi(z) = z^p$ has fixed points exactly the elements in $\mathbb{Z}_p$.  Thus
\begin{align*}
  \phi(x^{t(p^a-1)/(p-1)}) = x^{tp(p^a-1)/(p-1)} = x^{t(p^a-1)} x^{t(p^a-1)/(p-1)}.
\end{align*}
Since $x^{q-1} = 1$, $x^{t(p^a-1)/(p-1)}$ is a fixed point so it is in $\mathbb{Z}_p$.
Also, since the multiplicative group of $\mathbb{F}_q$ is cyclic, the elements $x^{t(p^a-1)/(p-1)}$ are distinct and there
are $p-1$ of them.
\end{proof} 

\begin{lemma}
  If $p$ is an odd prime, then $\alpha(G^+(p^2,p)) \geq \left\lfloor p^2/2 \right\rfloor$.
\end{lemma}

\begin{proof} 
$q = p^2$, $t = p$, so $n = p^2 (p-1)$.  Thus $\frac{1}{2} n^{2/3} \leq \frac{1}{2} p^2 = \frac{1}{2} q$.

The field $\mathbb{F}_q$ is $\mathbb{Z}_p[x]/(f(x))$, where $f(x)$ is some irreducible polynomial of degree $2$.
Thus elements
of $\mathbb{F}_q$ can be written as $\alpha x + \beta$ for $\alpha,\beta \in \mathbb{Z}_p$.  Since $t = p$, we need $H$ to be 
an additive subgroup of $\mathbb{F}_q$ of order $p$.  The additive subgroup generated by $x$ has order $p$,
so let $H = \left\{ 0,x,2x,3x,\dots,(p-1)x \right\}$.
We now claim the following set is an independent set:
\begin{align*}
  \left\{ (\bar{0},x^{2k}) : 0 \leq k < q/2 \right\}.
\end{align*}
Consider two vertices in this set: $(\bar{0},x^{2j})$ and $(\bar{0},x^{2k})$.
These will be adjacent if $x^{2j+2k} \in \bar{0} = H$, in other words $x^{2j+2k-1} \in \mathbb{Z}_p$.
But from Lemma~\ref{frobenius}, the powers of $x$ which give elements in
$\mathbb{Z}_p$ are of the form $t(p+1)$ for some $t$.
Since $p$ is an odd prime, $p+1$ is even.
Thus $x^{2j+2k-1} \notin \mathbb{Z}_p$.
\end{proof} 

Most likely, the above proof can be extended to $G^+(p^a,p^b)$ when $b$ divides $a$ as follows.  Let $q = p^a$ and view the field $\mathbb{F}_{q}$
as an extension field over $\mathbb{F}_p$; the Galois group $Gal(\mathbb{F}_q/\mathbb{F}_p)$ is cyclic
of order $a$ with generator the Frobenius automorphism used in Lemma~\ref{frobenius}.  Since $b$ divides $a$,
there is a subgroup of $Gal(\mathbb{F}_q/\mathbb{F}_p)$ of order $a/b$.  By the fundamental theorem of Galois theory,
this corresponds to an intermediate field extension of order $p^b$.  Thus we have a subfield of $\mathbb{F}_q$ of
order $p^b$ and an automorphism $\phi$ which fixes this subfield.  Replace the Frobenius automorphism in the above proof
by this $\phi$, investigate which powers of $x$ are fixed by $\phi$, and find a set whose sums avoid these powers
of $x$ to construct an independent set in $G^+(p^a,p^b)$.

\section{Conclusion and open problems} 
\label{secopenprob}

\begin{itemize}
  \item Looking at Table~\ref{ourtable}, it is somewhat strange that when $m$ is around $\log^2 t$
    the best lower bound switches from a simple construction (the Tur\'{a}n Graph) to the random
    graph.  Perhaps some combination of these two constructions could provide a good lower bound
    when $m$ is around $\log^2 t$.  Unfortunately, the two simple ideas do not work.  One option is
    to take $\ell$ random graphs forbidding $K_{2,t}$ and independence number $m/\ell$ as one color
    and all edges between the random graphs as the second color.  Another option is to take $\ell$
    cliques in red (of some size smaller than $t+1$) and put a random graph between cliques.  We are
    unable to make either of these two constructions beat the bounds in Table~\ref{ourtable}, even
    for a restricted range of $m$.

  \item The ideas in this paper can be extended to $r_k(K_{s,t};K_m)$ when $s$ is fixed using field
    norms, similar to the projective norm graphs.  Let $N : \mathbb{F}_{q^s} \rightarrow
    \mathbb{F}_{q}$ be the field norm of the extension of $\mathbb{F}_{q^s}$ over $\mathbb{F}_q$.
    (When $q$ is prime $N(x) = x^{(q^s-1)/(q-1)}$ and when $q$ is a prime-power the field norm is
    more complicated.)  Given $q$, $t$, and $s$, let $H$ be an additive subgroup of $\mathbb{F}_q$
    of order $t$ and form a graph $G^+$ as follows.  The vertex set is $\left( \mathbb{F}_q/H
    \right) \times \mathbb{F}^*_{q^s}$ and two vertices $(\bar{a},x)$ and $(\bar{b},y)$ are adjacent
    if $N(xy) \in \overline{a+b}$.  The graph $G^{\times}$ can be similarly extended using norms.
    These constructions will now avoid $K_{s,t}$ when $t \geq (s-1)! + 1$.  Using ideas from
    \cite{ee-szabo03}, the computations in Section~\ref{sec:spectrum} can be extended to find the
    spectrum of $G^+$ and $G^{\times}$.  Theorem~\ref{alonlower} can then be used to prove a lower
    bound on $r_k(K_{s,t};K_m)$ when $k \geq 2$ and $s$ is fixed.
\end{itemize}

\bibliographystyle{abbrv}
\bibliography{refs}

\begin{onlyarxiv}
\appendix
\section{Density of the Prime Numbers} 
\label{sec:primedensity}

In this appendix, we prove Lemma~\ref{primedensity}.  For convenience, we restate the lemma here.

\newtheorem*{primedensitylem}{Lemma~\ref{primedensity}}
\begin{primedensitylem}
  Fix integers $s, L \geq 1$.  There exists a constant $\delta > 0$ depending only on $s$ and $L$
  such that the following holds.  For every $t \geq 2$ and $m \geq 4^s L \log^{s}
  t$, either $\frac{\delta m^2 t}{L^2\log^{2s}(mt)} \leq 2$ or
  there is a prime power $q$ so that $q \equiv 1 \pmod{t}$ and
    \begin{align*}
       \delta \frac{m^2 t}{L^2 \log^{2s}(mt)} \leq \frac{q(q-1)}{t} \leq \frac{m^2 t}{L^2\log^{2s} (mt)}.
    \end{align*}
\end{primedensitylem}

Dirichlet's Theorem states that if $gcd(t,a) = 1$ then there are infinitely many prime numbers $p$
with $p \equiv a \pmod t$ so there are infinitely many prime numbers congruent to one modulo $t$.
This isn't quite enough for us since we need to find a prime in a specific range, but the prime
number theorem for arithmetic progressions states more than Dirichlet's theorem; essentially it says
that the primes are asymptotically equally divided modulo $t$ into the $\phi(t)$ congruence classes
coprime to $t$, where $\phi(t)$ is the Euler totient function.

\begin{thm} \label{thmsiegel} (Prime Number Theorem in Arithmetic Progressions)
Let $\pi(x;t,a)$ be the number of primes less than or equal to $x$ and congruent to $a$ modulo
$t$.  Then
\begin{align*}
\pi(x;t,a) = (1+o_t(1)) \frac{1}{\phi(t)} \frac{x}{\log x}.
\end{align*}
\end{thm}

The subscript of $t$ on $o$ implies the constant in the definition of $o$ can depend only on $t$.
In particular, when $t$ gets big there are primes congruent to $1 \pmod t$ between $(\ell - 0.01)t$ and
$\ell t$.

\begin{cor} \label{cor:primesbetween}
  There exists an absolute constant $T_0$ so that if $t \geq T_0$ and $\ell > 1.01$, then there
  exists a prime congruent to one modulo $t$ between $\ell t$ and $(\ell - 0.001)t$.
\end{cor}

Note that both Theorem~\ref{thmsiegel} and Corollary~\ref{cor:primesbetween} are not the best known
results of this kind, but are (more than) enough for our purposes.  For example, the requirement
that $\ell > 1.01$ in Corollary~\ref{cor:primesbetween} is an easy way to overcome the fact that
$\phi(t)$ can be as large as $t-1$ and to have at least one prime, $x/\log x$ from Theorem~\ref{thmsiegel}
must be at least $(1+o(1))\phi(t)$.  Requiring $x \geq 1.01 t$ and $t \geq T_0$ easily implies $x
\gg \phi(t)$.

Before the proof of Lemma~\ref{primedensity}, we get some computations out of the way.

\begin{lemma} \label{lem:lbiggerthanone}
  If $m,L,s,t \geq 1$ are real numbers, then there exists a constant $M_0$
  depending only on $s$ and $L$ so that if $m \geq M_0$ and $m \geq 4^s L \log^s t$, then
  \begin{align*}
    \frac{m}{L\log^{s}(mt)} > 1.01.
  \end{align*}
\end{lemma}

\begin{proof} 
Pick $M_0$ large enough so that for $m \geq M_0$, $2^s \log^s m < \frac{m}{1000 L}$.  Then
\begin{align*}
  \log^s(mt) &= \left( \log m + \log t \right)^s \leq \left( 2\max\{\log m,\log t\} \right)^s 
       = 2^s \max\{\log^s m, \log^s t\} \\
  &\leq 2^s \log^s m + 2^s \log^s t < \frac{m}{1000 L} + \frac{m}{2^s L} 
       \leq \frac{m}{1000 L} + \frac{m}{2L} \leq \frac{m}{1.01 L}.
\end{align*}
Therefore,
\begin{align*}
  \frac{m}{L\log^s (mt)} > \frac{m}{L (m/1.01L)} = 1.01.
\end{align*}
\end{proof} 

\begin{proof}[Proof of Lemma~\ref{primedensity}] 
For notational convenience, define $\ell = \frac{m}{L \log^s(mt)}$.  To prove the lemma, we must
produce a $\delta > 0$ so that for any $t \geq 2$ and $m \geq 4^s L \log^s t$, either $\delta
\ell^2 t \leq 2$ or there exists a prime power $q$ so that $q \equiv 1 \pmod t$ and $\delta \ell^2 t
\leq q(q-1)/t \leq \ell^2 t$.

Let $T_0$ and $M_0$ be the constants from Corollary~\ref{cor:primesbetween} and
Lemma~\ref{lem:lbiggerthanone} respectively, and define $T_1$ so that $M_0 = 4^s L \log^s T_1$.  The
constants $T_0$, $T_1$, and $M_0$ depend only on $s$ and $L$.  Define $\delta$ small
enough so that the following equations are satisfied:
\begin{align*}
  \frac{\delta M_0^2 T_1}{L^2\log^{2s}(M_0 T_1)} \leq 2, \quad \delta(1.01 T_0)^2 T_0 \leq 2, \quad \delta
  < \frac{1}{16}.
\end{align*}
The definition of $\delta$ depends only on $s$ and $L$ as required.

Assume that $\delta \ell^2 t > 2$.  We must now find a prime power $q$ so that $q \equiv 1 \pmod
t$ and $\delta \ell^2 t \leq q(q-1)/t \leq \ell^2 t$.  Multiplying everything by $t$ and taking
the square root, we must find $q$ between
\begin{align} \label{eq:rangeofqqm1}
  \sqrt{\delta} \ell t \leq \sqrt{q(q-1)} \leq \ell t.
\end{align}
$\sqrt{q(q-1)}$ is approximately $q$; in fact, if we can find $q$ in the following range
\begin{align} \label{eqA:rangeofq}
  2\sqrt{\delta} \ell t \leq q \leq \ell t,
\end{align}
then \eqref{eq:rangeofqqm1} will be satisfied.  This is because
\begin{align*}
  \sqrt{q(q-1)} = \sqrt{q}\sqrt{q-1} \geq \sqrt{q} \cdot \frac{\sqrt{q}}{2} = \frac{q}{2},
\end{align*}
so if we find $q \geq 2\sqrt{\delta} \ell t$, then $\sqrt{q(q-1)} \geq q/2 \geq \sqrt{\delta} \ell
t$ so that \eqref{eq:rangeofqqm1} is satisfied.

We now divide into cases depending on if $t \geq T_0$ or $m \geq M_0$.

\begin{itemize}
  \item \textbf{Case 1: $m \geq M_0$ and $t \geq T_0$}: Lemma~\ref{lem:lbiggerthanone} shows $\ell >
    1.01$ and Corollary~\ref{cor:primesbetween} then shows there is a prime $q$ congruent to one
    modulo $t$ between $(\ell - 0.001)t$ and $\ell t$.  Since $\delta < \frac{1}{16}$, $2
    \sqrt{\delta}\ell < \ell - 0.001$.  We have now found $q$ in the range from \eqref{eqA:rangeofq}.

  \item \textbf{Case 2: $m < M_0$}: By assumption, $m \geq 4^s L \log^s t$.  Thus $m < M_0$ and the
    definition of $T_1$ shows that $t \leq T_1$.  But then, 
    \begin{align*}
      \delta \ell^2 t \leq \frac{\delta M_0^2 T_1}{L^2 \log^{2s}(M_0T_1)} \leq 2
    \end{align*}
    by the definition of $\delta$, and this contradicts that $\delta \ell^2 t > 2$.

  \item \textbf{Case 3: $m \geq M_0$ and $t < T_0$ and $\ell/T_0 > 1.01$}:  Let $t' = tT_0$ so $t'
    \geq T_0$ and $\ell' = \ell/T_0 > 1.01$.  Corollary~\ref{cor:primesbetween} show that there
    exists a prime $q$ congruent to one modulo $t'$ between $(\ell' - 0.001)t'$ and $\ell' t'$.  That
    is,
    \begin{align*}
      \left(\frac{\ell}{T_0} - 0.001\right) t T_0 \leq q \leq \frac{\ell}{T_0} \cdot t T_0 = \ell t.
    \end{align*}
    We now want to show that $q$ is in the range \eqref{eqA:rangeofq}.  
    In other words, show
    \begin{align*}
      2\sqrt{\delta}\ell &< \left( \frac{\ell}{T_0} - 0.001 \right) T_0 \\
      2 \sqrt{\delta} \cdot \frac{\ell}{T_0} &< \frac{\ell}{T_0} - 0.001.
    \end{align*}
    Written this way, we can easily see that since $\delta < 1/16$, this inequality is true since
    $\ell/T_0 > 1.01$.  Lastly, $q$ congruent to one modulo $t' = tT_0$ implies $q$ is congruent to
    one modulo $t$, so we have found $q$ with the required properties.

  \item \textbf{Case 4: $t < T_0$ and $\ell/T_0 < 1.01$}:  In this case, $t < T_0$ and $\ell < 1.01
    T_0$ implies
    \begin{align*}
      \delta \ell^2 t \leq \delta (1.01 T_0)^2 T_0 \leq 2
    \end{align*}
    by the definition of $\delta$, but this contradicts that $\delta \ell^2 t > 2$.
\end{itemize}
\end{proof} 

\section{Lower bounds on $r_k(K_{2,t};K_m)$ for $k \geq 3$} \label{appendixcalcs} 

In this appendix, we sketch the proof that inequality \eqref{eq:alonlower} in Theorem~\ref{alonlower} is
true when $d = \sqrt{nt}$, $\lambda = (nt)^{1/4}$, and $m = 2k \sqrt{n/t} \log n$.  In the
computations to follow, let $\theta = \sqrt{n/t} \log n$ which will simplify the notation.  The
inequality \eqref{eq:alonlower} is (temporarily disregard the constants)
\begin{align*}
  \left( \frac{md^2}{\lambda n \log n}\right)^{\frac{2kn\log n}{d}} \left( \frac{\lambda n}{md} \right)^{km} \left(\frac{m}{n} \right)^{m(k-1)} < 1.
\end{align*}
Substituting $d = \sqrt{nt}$ and $\lambda = (nt)^{1/4}$, this simplifies to
\begin{align*}
  \left( \frac{m nt}{(nt)^{1/4} n \log n}\right)^{\frac{2k\sqrt{n}\log n}{\sqrt{t}} }
  \left( \frac{(nt)^{1/4} n}{m\sqrt{nt}} \right)^{km} \left(\frac{m}{n} \right)^{m(k-1)} < 1.
\end{align*}
Simplifying, this is
\begin{align*}
  \left( \frac{m t^{3/4}}{n^{1/4} \log n}\right)^{2k\theta}
  \left( \frac{n^{3/4}}{m t^{1/4}} \right)^{km} \left(\frac{m}{n} \right)^{m(k-1)} < 1.
\end{align*}
Substitute in $m = 2k\theta$:
\begin{align*}
  \left( \frac{\theta t^{3/4}}{n^{1/4} \log n}\right)^{2k\theta}
  \left( \frac{n^{3/4}}{\theta t^{1/4}} \right)^{2k^2\theta} \left(\frac{\theta}{n} \right)^{2k\theta(k-1)} < 1.
\end{align*}
Drop a $2k\theta$ in the exponent, and substitute in $\theta = \sqrt{n/t}\log n$:
\begin{align*}
  \left( n^{1/4}t^{1/4}\right)
  \left( \frac{n^{1/4} t^{1/4}}{\log n} \right)^{k} \left(\frac{\log n}{n^{1/2}t^{1/2}} \right)^{(k-1)} < 1.
\end{align*}
Simplify to
\begin{align*}
  (nt)^{\frac{1}{4} + \frac{k}{4} - \frac{k-1}{2}} \log^{-1} n < 1.
\end{align*}
When $k \geq 3$, the exponent on $nt$ is non-positive so the expression is true (even when we add back in
the constants that got dropped.)

Thus we can conclude that for $k \geq 3$ and $m = 2k\theta = 2k \sqrt{n/t}\log n$,
$r_k(K_{2,t};K_m) > n$.
Solving for $n$ in terms of $m$ we obtain $r_k(K_{2,t};K_m) = \Omega(m^2t/\log^2(mt))$,
proving \refkthreelower.

\end{onlyarxiv}
\end{document}